\documentclass[a4paper,12pt]{amsart}

\usepackage{amssymb}
\usepackage{amsmath}
\usepackage{amsthm}

\usepackage[colorlinks]{hyperref}
\usepackage{hyperref}
\renewcommand\eqref[1]{(\ref{#1})} 

\allowdisplaybreaks \numberwithin{equation}{section}

\theoremstyle{plain}
\newtheorem{theorem}{Theorem}[section]

\newtheorem{lemma}[theorem]{Lemma}

\theoremstyle{definition}
\newtheorem{definition}[theorem]{Definition}
\newtheorem{remark}[theorem]{Remark}

\begin{document}

\title[Integro-differential diffusion equations]{Integro-differential diffusion equations on graded Lie groups}

\author[J.E. Restrepo, M. Ruzhansky and B.T. Torebek]{Joel E. Restrepo$^*$, Michael Ruzhansky and Berikbol T. Torebek}
\address{\textcolor[rgb]{0.00,0.00,0.84}{Joel E. Restrepo \newline Department of Mathematics: Analysis,
Logic and Discrete Mathematics \newline Ghent University, Krijgslaan 281, Ghent, Belgium \newline and \newline Department of Mathematics \newline CINVESTAV, Mexico City, Mexico}}
\email{\textcolor[rgb]{0.00,0.00,0.84}{joel.restrepo@cinvestav.mx}}
\address{\textcolor[rgb]{0.00,0.00,0.84}{Michael Ruzhansky \newline Department of Mathematics: Analysis,
Logic and Discrete Mathematics \newline Ghent University, Krijgslaan 281, Ghent, Belgium \newline
 and \newline School of Mathematical Sciences \newline Queen Mary University of London, United Kingdom}}
\email{\textcolor[rgb]{0.00,0.00,0.84}{michael.ruzhansky@ugent.be}}
\address{\textcolor[rgb]{0.00,0.00,0.84}{Berikbol T. Torebek \newline Department of Mathematics: Analysis,
Logic and Discrete Mathematics \newline Ghent University, Krijgslaan 281, Ghent, Belgium \newline and  \newline Institute of
Mathematics and Mathematical Modeling \newline 125 Pushkin str.,
050010 Almaty, Kazakhstan}}
\email{\textcolor[rgb]{0.00,0.00,0.84}{berikbol.torebek@ugent.be}}

\dedicatory{}
\let\thefootnote\relax\footnote{$^{*}$Corresponding author}
\date{\today}
\thanks{This research has been funded by the FWO Odysseus 1 grant G.0H94.18N: Analysis and Partial Differential Equations and by the Methusalem programme of the Ghent University Special Research Fund (BOF) (Grant number 01M01021). MR is also supported by EPSRC grant EP/R003025/2 and EP/V005529/1. BT is also supported by the Science Committee of the Ministry of Education and Science of the Republic of Kazakhstan (Grant No. AP14869090).
\\
The general content of the manuscript, in whole or in part, is not submitted, accepted, or published elsewhere, including conference proceedings.
}

\begin{abstract}
We first study the existence, uniqueness and well-posedness of a general class of integro-differential diffusion equation on $L^p(\mathbb{G})$ $(1<p<+\infty$, $\mathbb{G}$ is a graded Lie group). We show the explicit solution of the considered equation. The equation involves a nonlocal in time operator (with a general kernel) and a positive Rockland operator acting on $\mathbb{G}.$ Also, we provide $L^p(\mathbb{G})-L^q(\mathbb{G})$ $(1<p\leqslant 2\leqslant q<+\infty)$ norm estimates and time decay rate for the solutions. In fact, by using some contemporary results, one can translate the latter regularity problem to the study of boundedness of its propagator which strongly depends on the traces of the spectral projections of the Rockland operator. Moreover, in many cases, we can obtain time asymptotic decay for the solutions which depends intrinsically on the considered kernel. As a complement, we give some norm estimates for the solutions in terms of a homogeneous Sobolev space in $L^{2}(\mathbb{G})$ that involves the Rockland operator. We also give a counterpart of our results in the setting of compact Lie groups. Illustrative examples are also given.
\end{abstract}



\maketitle

\tableofcontents

\section{Introduction}

The Abel integral equation with kernel $k(t)=t^{\alpha-1}$ $(0<\alpha<1)$ played a very important role in the construction of the well-known family of Riemann-Liouville fractional integro-differential operators, see e.g. for some historical facts and developments \cite[Chapter 1, Section 2]{samko}. Later, challenges for these type of equations came from the consideration of more general kernels. To the best of our knowledge, the first approach towards solving these equations with general kernels was by Sonine \cite{Sonine}. In the last years, the topic has became more popular and several researchers have been investigating some special classes of these kernels. Let us mention some of the contributions made by Cl\'ement and Nohel \cite{Clement,[50]}, Kochubei \cite{Kochubei}, Luchko and Yamamoto \cite{Luchko1}, Vergara and Zacher \cite{Vergara1}, and others. Notice that from the point of view of evolutionary integral equations involving general classes of kernels, Pr\"{u}ss can serve as a reference with his book \cite{Prus}.

The main purpose of this paper is to study the following integro-differential diffusion equation over the space $L^p(\mathbb{G})$ $(1<p<+\infty$, $\mathbb{G}$ is a graded Lie group):
\begin{equation}\label{1}
\partial_t\left(k\ast\left[u-u_0\right]\right)+\mathcal{R}u=f,\quad 0<t\leqslant T<+\infty,
\end{equation}
with initial data
\begin{equation}\label{2}
u_0\in C(\mathbb{G})\cap L^p(\mathbb{G}),
\end{equation}
where $k$ is a scalar kernel $\neq0$ in $L^1_{loc}(\mathbb{R}_+)$, $\mathcal{R}$ is a positive (unbounded) Rockland operator on $\mathbb{G}$ of homogeneous order $\nu$, $f\in C([0,T],L^p(\mathbb{G}))$ and $(k\ast u)(t)$ denotes the Laplace convolution, i.e.
$$(k\ast u)(t)=\int\limits_0^tk(t-s)u(s){\rm d}s.$$
It is important to mention that we are considering the operator $$\mathcal{R}:\mathcal{D}(\mathbb{G})\subset L^p(\mathbb{G})\to L^p(\mathbb{G}),$$ whose dense (in $L^p(\mathbb{G})$) domain $\mathcal{D}(\mathbb{G})$ is the the space of smooth functions compactly supported in $\mathbb{G},$ see \cite[Subsection 4.3.1]{FR16} for more details and specifically \cite[Theorem 4.3.3]{FR16}.  

Let us now recall the class of kernels that we shall use everywhere in this manuscript. The following class of kernels was introduced in \cite{[50],Clement}: 

We assume that $k\in L^1_{loc}(\mathbb{R}_+),$ is nonnegative and nonincreasing, and there exists a nonnegative kernel $l\in L^1_{loc}(\mathbb{R}_+)$ such that $$(k\ast l)(t)=1,\quad t>0.\eqno(\mathcal{PC})$$
It has to be highlighted that a kernel satisfying condition ($\mathcal{PC}$) was first considered by Sonine in \cite{Sonine}. Consequently, integro-differential operators with kernels $(k, l)\in \mathcal{PC}$ are usually called operators with Sonine kernels.

\textcolor{red}{Note that the model \eqref{1} in $\mathbb{G}=\mathbb{R}^n$ with $\mathcal{R}$ being the Laplacian arises in studying the processes of anomalous and ultraslow diffusions \cite{[30],ultra-2,ultra-1}. From the point of view of applications in physics, researchers have focused on considering the nonlocal in time
reaction-diffusion equations. In some cases, the time-fractional derivative associated to the considered problem describes the process
of anomalous diffusion, dynamic processes in materials with memory, diffusion in
fluids in porous media with memory, etc; we recommend to see the following sources in this direction \cite{book-1,[18],[15]}. More specialized works in mathematics by using this type of kernels can be found e.g. in \cite{[3],[17],[19]}.}

\textcolor{red}{Having in mind the ideas of the last paragraph and the general nature of graded Lie groups, we aim in this work to gather distinct behaviour and objects coming from the time integro-differential operator and the large class of space-operators that can be considered in this type of Lie group (e.g. the Rockland ones). In our context, an important class of problems appear when one deals with hypoelliptic operators $\mathfrak{L}$. Some of the most interesting ones are when $\mathfrak{L}$ is a H\"ormander sum of squares. By Rothschild and Stein \cite{RS}, one is led to considering a lifting of $\mathfrak{L}$ to a sub-Laplacian operator on a general stratified group. While, for higher order operators, one is led to considering Rockland operators.}

\medskip We now start by discussing the solution of problem \eqref{1}-\eqref{2} $(f\equiv0)$ in $L^p(\mathbb{G})$ for $1<p<+\infty$. The general case $(f\neq0)$ will be treated in the main results in  Section \ref{main}. Thus, we alternatively study the equivalent integral equation for \eqref{1}: 
\begin{equation}\label{integral-equivalent}
u(t)+(l\ast \mathcal{R}u)(t)=u_0. 
\end{equation}
For proving that \eqref{integral-equivalent} implies \eqref{1}, we additionally need to assume that $l\in BV_{loc}(\mathbb{R}_+)$ (strongly bounded variation functions). It is important to highlight that everywhere below we will use some results from the book \cite{Prus} where for the space $BV_{loc}(\mathbb{R}_+)$ the author normalized the functions $k$ in this class by assuming $k(0)=0$ and $k(\cdot)$ is left-continuous on $\mathbb{R}_+.$ Nevertheless, in our case, it is not necessary to assume it. Also, we have verified that the additionally conditions are not necessary to use some of the results from \cite{Prus}.

We now recall briefly some important notions about the solution of the integral equation \eqref{integral-equivalent} given by J. Prüss in \cite{Prus}. Here we just adapt the concepts by using the Rockland operator $\mathcal{R}$ in the space $L^p(\mathbb{G}).$ Below, we usually equipped the space $\mathcal{D}(\mathbb{G})$ with the graph norm $\|\cdot\|_{(L^p(\mathbb{G}),\mathcal{R})}$, i.e. $$\|v\|_{(L^p(\mathbb{G}),\mathcal{R})}=\|v\|_{L^p(\mathbb{G})}+\|\mathcal{R}v\|_{L^p(\mathbb{G})}$$ for each $v\in \mathcal{D}(\mathbb{G}).$ By \cite[Theorem 4.4.3]{FR16} this is the Sobolev space $L^p_\nu(\mathbb{G})$, where $\nu$ is the homogeneous order of $\mathcal{R}.$  

\begin{definition}
Let $u$ be a function of the space $C([0,T];L^p(\mathbb{G}))$. It is said that $u$ is:
\begin{enumerate}
\item {\it strong solution} of \eqref{integral-equivalent} on $[0,T]$ if $u\in C([0,T];\mathcal{D}(\mathbb{G}))$ and equation \eqref{integral-equivalent} holds for all $t\in [0,T]$;
\item {\it mild solution} of \eqref{integral-equivalent} on $[0,T]$ if $l\ast u\in C([0,T];\mathcal{D}(\mathbb{G}))$ and 
\[
u(t)=u_0-\mathcal{R}\big(l\ast u\big)(t)
\]
for all $t\in [0,T]$;
\item {\it weak solution} of \eqref{integral-equivalent} on $[0,T]$ if 
\[
\langle u(t),y\rangle =\langle u_0,y\rangle -\langle (l\ast u)(t),\mathcal{R}y\rangle,
\]
for all $t\in [0,T]$ and each $y\in \mathcal{D}(\mathbb{G}).$
\end{enumerate}
\end{definition}
Notice that each strong solution of \eqref{integral-equivalent} is a mild solution, and every mild solution is a weak one. Below we discuss the well-posedness of \eqref{integral-equivalent}, which is an extension of the classical notion of well-posed Cauchy problems.

\begin{definition}
Equation \eqref{integral-equivalent} is said to be well-posed if the following two conditions are satisfied: 
\begin{enumerate}
    \item For each $v\in\mathcal{D}(\mathbb{G})$ there exists a unique strong solution $u_v(t)$ of 
    \[
  u(t)=v+(l\ast \mathcal{R}u)(t),\quad t\geqslant0.  
    \]
    \item For all sequences $\{v_{m}\}_{m\geqslant1}\subset\mathcal{D}(\mathbb{G})$, such that $v_{m}\to0$ as $m\to+\infty$, imply $u_{v_m}(t)\to0$ as $m\to+\infty$ in $L^p(\mathbb{G})$, uniformly on $[0,T].$
\end{enumerate} 
\end{definition}

Assume that \eqref{integral-equivalent} is well-posed, we can then define the solution operator $S(t)$ (we always assume exponentially bounded) for \eqref{integral-equivalent} as follows:
\[
S(t)v=u_v(t),\quad v\in\mathcal{D}(\mathbb{G}),\quad t\geqslant 0.
\]
Without loss of generality, we frequently denote $u_v(t)$ by $u(t).$ It is known that a well-posed problem \eqref{integral-equivalent} admits a solution operator $S(t)$ (also it is called resolvent or propagator), and vice versa \cite[Proposition 1.1]{Prus}. 

By \cite[Theorem 4.2]{Prus} and \cite[Corollary 4.2.9]{FR16}, the integral equation \eqref{integral-equivalent} admits a resolvent in $L^p(\mathbb{G})$ (exponentially bounded) whenever the initial condition $u_0$ is continuous in $L^p(\mathbb{G})$. Therefore, by \cite[Proposition 1.1]{Prus}, equation \eqref{integral-equivalent} is well-posed.   

Since $l\in BV_{loc}(\mathbb{R}_+)$ we have that the resolvent is differentiable (see \cite[Page 34]{Prus}). So, the solution in $L^p(\mathbb{G})$ of problem \eqref{1}-\eqref{2} can be rewritten as $u(t)=S(t)u_0,$ where $S(t)$ is the propagator, see e.g. \cite[Chapter 2]{Prus}. This form of the solution will be useful to calculate the $L^p(\mathbb{G})-L^q(\mathbb{G})$ boundedness of solutions. This will be discussed in detail in Section \ref{main}.

In this work, we use the class of kernels $\mathcal{PC}$ \cite{[50],Clement} which allows us to obtain a nonnegative and nonincreasing solution for the one dimensional case of our considered integro-differential equation. The latter statement will be significant and essential to prove the $L^p(\mathbb{G})-L^q(\mathbb{G})$ norm estimates for the solution of equation \eqref{integral-equivalent}. Indeed, we will see that the regularity is mainly associated with the study of spectral and Fourier (H\"ormander type) multipliers. In general, by using some recent results on Fourier multipliers \cite{RR2020,SRR} that can be applied to obtain spectral multipliers, it can be showed that the norm estimates for the solution of our integro-differential difussion equation \eqref{1} can be reduced to the time asymptotics of its propagator in the noncommutative Lorentz space norm \cite{[51]} (the space is associated with a semifinite von Neumann algebra \cite{von,[46],von2}), which involves calculating the trace of the spectral projections of the Rockland operator $\mathcal{R}$ \cite{david}. Moreover, we are able to establish asymptotic time behavior for the solution of equation \eqref{1} that will depend on the considered kernel. Notice that the first result in this direction, we mean in Lie groups, was given on compact groups in \cite{wagner}, where one considered the particular case of a kernel in the form $k(t)=t^{\alpha-1}$ $(0<\alpha<1)$. Our work somehow extends previous results on $\mathbb{R}^n$ \cite{Vergara1} where the boundedness of the solution of a nonlocal in time diffusion equation was provided in different spaces of $L^p(\mathbb{R}^n)$ $(1<p<+\infty).$ 

\textcolor{red}{It is necessary to mention that all the results established in this paper are new in the context of graded and compact Lie groups. One the one hand, we frequently contrast our results with those already known in the classical model of $\mathbb{G}=\mathbb{R}^n$. In this case, usually, our results coincide with them. On the other hand, we also provide some examples of particular kernels and equations where it seems that the estimates are new in this setting and even in $\mathbb{R}^n,$ see e.g. Section \ref{comparison}.}    

The structure of the paper is a follows: In Section \ref{preliminaries} we recall several notions, definitions and known results about graded Lie groups and the corresponding Fourier analysis, and Von Neumann algebras. These will be the basis to provide the principal statements in the next section. In Section \ref{main}, we establish the main results on the well-posedness in $L^p(\mathbb{G})$ $(1<p<+\infty)$, $L^p(\mathbb{G})-L^q(\mathbb{G})$ $(1<p\leqslant 2\leqslant q<+\infty)$ boundedness and asymptotic time behaviour of the solution for equation \eqref{integral-equivalent}. Here we also provide some norm estimates for the solutions of problem \eqref{1}-\eqref{2} in terms of a homogeneous Sobolev space in $L^{2}(\mathbb{G})$ that involves the Rockland operator. In Section \ref{roc}, examples of Rockland operators are given. Some illustrative examples and comparison results are provided in Section \ref{comparison}.      

\section{Preliminaries}\label{preliminaries}
\medskip Below we collect some basic results about the one dimensional equation \eqref{1} that we need to prove the main results of this paper. Also, we recall some notions and statements about graded Lie groups and the corresponding Fourier analysis. In the last subsection, we provide several definitions on von Neumann algebras and the trace of an operator. The latter information will be an important tool in the analysis of the norm estimates for the solution of problem \eqref{1}-\eqref{2}. In fact, it will be essential to measure the contribution of the Rockland operator in our estimates.    

\subsection{One dimensional integro-differential equation} \label{oned}

Let us consider the following integro-differential equation with Sonine kernel
\begin{equation}\label{1.1}
\partial_t\left(k\ast\left[s_\mu-1\right]\right)(t)+\mu s_\mu(t)=0,\quad t>0,\,\, s_\mu(0)=1,\,\mu>0.
\end{equation}
Note that equation \eqref{1.1} is equivalent to the Volterra equation
\begin{equation}\label{1.2}
s_\mu(t)+\mu(l\ast s_\mu)(t)=1,\quad t\geqslant 0.
\end{equation}
It is known that (see \cite[Prop. 4.5]{Prus}) the assumption $(k, l)\in \mathcal{PC}$ implies that the solution of \eqref{1.1}, i.e. $s_\mu(t)$, it is nonnegative and nonincreasing.

Next we recall some important estimates for the function $s_\mu(t)$ that were proved in \cite[Lemma 6.1]{Vergara1}.

\begin{lemma}\label{lemma-decay} 
Let $(k, l)\in \mathcal{PC}$ and  $\mu>0.$ Then
\begin{equation}\label{1.3}
\frac{1}{1+\mu k(t)^{-1}}\leqslant s_\mu(t)\leqslant \frac{1}{1+\mu(1\ast l)(t)},\quad\textrm{almost all}\quad t>0.\end{equation}
\end{lemma}

\subsection{Graded Lie groups and their Fourier analysis}\label{sub.int.graded}

Now we give some of the basis notions and properties of graded Lie groups that will be implicitly used through the whole paper. Full details can be found e.g. in \cite{FR16,FS82}.

A connected simply connected Lie group $\mathbb{G}$ is called a \textit{graded Lie group} if its Lie algebra $\mathfrak{g}$ can be endowed with a vector space decomposition $\mathfrak{g}= \bigoplus_{i=1}^{\infty}\mathfrak{g}_{i}\,,$ such that all, but finitely many $\mathfrak{g}_{i}$'s, are  $\{0\}$ and $[\mathfrak{g}_{i},\mathfrak{g}_{j}]\subset \mathfrak{g}_{i+j}$.

By $\widehat{\mathbb{G}}$ we denote the unitary dual of $\mathbb{G}$; that is the set of all equivalence classes of irreducible, strongly
continuous and unitary representations of $\mathbb{G}$.

We now introduce the so-called \textit{Rockland operators}. These operators are frequently denoted by $\mathcal{R}$, and they are left-invariant differential operators, that are homogeneous of positive degree \cite[Def. 3.1.15]{FR16}, and satisfy the \textit{Rockland condition} \cite[Def. 4.1.1]{FR16}, hence they are hypoelliptic. 

The \textit{group Fourier transform} is defined on $L^1(\mathbb{G},{\rm d}x)$ by 
\[
 \mathcal{F}_{\mathbb{G}}f(\pi)\equiv \widehat{f}(\pi) \equiv \pi(f):= \int_{\mathbb{G}}f(x)\pi(x)^{*}\,{\rm d}x\,,
\]
where ${\rm d}x$ stands for the (bi-invariant) Haar measure on $\mathbb{G}$. 

The Fourier transform satisfies the following property:  For $f \in \mathcal{S}(\mathbb{G}) \cap L^1(\mathbb{G})$ and $X \in \mathfrak{g}$
\[
 \mathcal{F}_{\mathbb{G}}(X f)(\pi)=\pi(X)\widehat{f}(\pi)\,.
\]
Consequently for the Rockland operator $\mathcal{R}$ on $\mathbb{G}$ we have
\[
 \mathcal{F}_{\mathbb{G}}(\mathcal{R} f)(\pi)=\pi(\mathcal{R})\widehat{f}(\pi)\,,
\]
and by using the infinitesimal representation of $\pi(\mathcal{R})$ (see \cite{HJL85}), the operator has the following matrix representation 
\begin{equation}
\label{matrix.repr}
\mathcal{F}_{\mathbb{G}}(\mathcal{R} f)(\pi)=\left\{ \pi_{i}^{2}\cdot \widehat{f}(\pi)_{i,j}\right\}_{i,j \in \mathbb{N}}\,.
\end{equation}
The orbit method (see e.g. \cite{CG90,Kir04}) gives us the possibility to describe $\widehat{\mathbb{G}}$ as the subset of some Euclidean space. This allows to equip $\widehat{\mathbb{G}}$ with a concrete measure, called in the literature the \textit{Plancherel measure} usually denoted by $\mu$. On the other hand for $f \in L^1(\mathbb{G}) \cap L^2(\mathbb{G})$, and for $x \in \mathbb{G}$, the operators $\pi(f)\pi(x)$, $\pi(x)\pi(f)$ and $\pi(f)$ are (under the same equivalent class $\pi$) trace class and Hilbert--Schmidt, respectively, and integrable against $\mu$. Under these considerations we have the isometry, known as the \textit{Plancherel formula}
\begin{equation}\label{plan.gr}\begin{split}
\int_{\mathbb{G}} |f(x)|^2\,{\rm d}x=\int_{\widehat{\mathbb{G}}}\textnormal{Tr}(\pi(f)\pi(f)^{*})\,{\rm d}\mu(\pi)=\int_{\widehat{{\mathbb{G}}}}\|\pi(f)\|^{2}_{\textnormal{HS}(\mathcal{H}_{\pi})}\,{\rm d}\mu(\pi)\,,\end{split}
\end{equation}
while any $f \in \mathcal{S}(\mathbb{G})$ may be recovered via the \textit{Fourier inversion formula} given by 
\begin{equation}\label{Four.inv.for}\begin{split}
f(x)=\int_{\widehat{\mathbb{G}}}\textnormal{Tr}(\pi(x)\pi(f))\,{\rm d}\mu(\pi)=\int_{\widehat{\mathbb{G}}}\textnormal{Tr}(\pi(f)\pi(x))\,{\rm d}\mu(\pi)\,.\end{split}
\end{equation}
It is important to mention that the orbit methods and its consequences hold true in the more general setting of a connected simply connected nilpotent Lie group.

\subsection{Von Neumann algebras and other definitions}\label{preli}

Let $\mathfrak{L}(\mathcal{H})$ be the set of linear operators defined on a Hilbert space $\mathcal{H}$. In this context, the concept of $\tau$-measurability on a von Neumann algebra $M$ (see e.g. \cite{von,[46],[47],von2}) and the spectral projections give us the possibility to approximate unbounded operators by bounded ones. 

For our analysis it is enough to set $M$ to be the right group von Neumann algebra $VN_R(\mathbb{G})$ where $\mathbb{G}$ is a graded Lie group.    

The group von Neumann algebra $VN_R(\mathbb{G})$ is generated by all the right actions of $\mathbb{G}$ on $L^2(\mathbb{G})$ ($\pi_R(g)f(x)=f(xg)$ with $g\in \mathbb{G}$), i.e. $$VN_R(\mathbb{G})=\{\pi_R(\mathbb{G})\}^{!!}_{g\in \mathbb{G}},$$ where $!!$ is the bicommutant of the subalgebra $\{\pi_R(g)\}_{g\in \mathbb{G}}\subset \mathfrak{L}(L^2(\mathbb{G}))$. This follows from \cite{von} since $$VN_R(\mathbb{G})^{!}=VN_L(\mathbb{G})\,\,\, \text{and}\,\,\, VN_L(\mathbb{G})^{!}=VN_R(\mathbb{G}),$$ where the symbol $!$ represents the commutant of the group. 

We now recall an important definition that will be used frequently in the development of this paper.   
\begin{definition}
Let $M$ be a von Neumann algebra. A trace on the positive part $M_+=\{L\in M: L^*=L>0\}$ of $M$ is a functional defined on $M_+$, taking non-negative, possibly infinite, real values, with the following properties:
\begin{enumerate}
    \item If $L\in M_+$ and $T\in M_+$ then $\tau(L+T)=\tau(L)+\tau(T)$;
    \item If $L\in M_+$ and $\gamma\in\mathbb{R}^{+}$ then $\tau(\gamma L)=\gamma\tau(L)$ (with $0\cdot+\infty=0$);
    \item If $L\in M_+$ and $U$ is an unitary operator on $M$ then $\tau(ULU^{-1})=\tau(L).$
\end{enumerate}
\end{definition}

\section{Main results}\label{main}

In this section, we present the principal results of the paper. Mainly, we analyze  the integral equation \eqref{integral-equivalent} (even when $f\neq0$) instead of problem \eqref{1}-\eqref{2} since they are equivalent under the considered kernels. We first establish the well-posedness in $L^p(\mathbb{G})$ $(1<p<+\infty)$ of \eqref{integral-equivalent} and show the explicit form of the solution. Also, we complement the study by proving that the solution is bounded in $L^q(\mathbb{G})$ with data in $L^p(\mathbb{G})$ for $1<p\leqslant2\leqslant q<+\infty.$ Moreover, we give the asymptotic behaviour of this solution. We additionally provide some norm estimates for the solution of problem \eqref{1}-\eqref{2} in a homogeneous Sobolev space of $L^2(\mathbb{G}).$       

\subsection{Well-posedness and $L^p-L^q$ estimates}
\begin{theorem}\label{integral-thm}
Let $(k,l)\in \mathcal{PC}$ be such that $l$ is positive and $l\in BV_{loc}(\mathbb{R}_+).$ Let $\mathbb{G}$ be a graded Lie group of homogeneous dimension $Q$. Let $\mathcal{R}$ be a positive Rockland operator of homogeneous degree $\nu$ on $\mathbb{G}$. If $u_0$ is a continuous function in $L^p(\mathbb{G})$ $(1<p<+\infty)$ then the integral equation \eqref{integral-equivalent} is well-posed. Also, the explicit solution can be written as follows:
\[
u(t,x)=\int_{\widehat{\mathbb{G}}}\textnormal{Tr}[\pi(x)\overline{s}_{\pi}(t)\widehat{u_0}(\pi)]\,{\rm d}\mu(\pi),\quad x\in\mathbb{G},
\]
where $\overline{s}_{\pi}(t)$ is defined in the proof.

Moreover, we get the following time decay rate for the solution in $L^q(\mathbb{G})$ $(2\leqslant q<+\infty)$ with any datum in $L^p(\mathbb{G})$ $(1<p\leqslant 2):$ 
\begin{equation}\label{lp}
\|u(t)\|_{L^q(\mathbb{G})}\lesssim \bigg((1\ast l)(t)\bigg)^{-Q/\nu\left(\frac{1}{p}-\frac{1}{q}\right)}\|u_0\|_{L^p(\mathbb{G})},\quad \frac{\nu}{Q}\geqslant\frac{1}{p}-\frac{1}{q}.    
\end{equation}
\end{theorem}
\begin{proof}
We first start by showing the explicit solution of problem \eqref{1}-\eqref{2} $(f\equiv0)$ by using the group Fourier transform. 

Hence, let us now apply the group Fourier transform in equation \eqref{1} $(f\equiv0)$ with respect to the variable $x$ for all $\pi\in\widehat{\mathbb{G}}$ and then
\begin{align}\label{fouriere}
\begin{split}
\partial_t\left(k\ast\left[\widehat{u}(t,\pi)-\widehat{u_0}(\pi)\right]\right)+\pi(\mathcal{R})\widehat{u}(t,\pi)=0, \quad \widehat{u}(0,\pi)=\widehat{u_0}(\pi);\quad t>0.
\end{split}
\end{align}
Taking into account the infinitesimal representation of $\pi(\mathcal{R})$, the functional calculus allows the latter equation to be seen  componentwise as an infinite system of equations of the form
\begin{equation}\label{diagon}
\begin{split}
\partial_t\left(k\ast\left[\widehat{u}(t,\pi)_{ij}-\widehat{u_0}(\pi)_{ij}\right]\right)+\pi_{i}^{2}\widehat{u}(t,\pi)_{ij}=0,\quad \widehat{u}(0,\pi)_{ij}=\widehat{u_0}(\pi)_{ij};\quad t>0.
\end{split}
\end{equation}
where we are considering any $i,j\in\mathbb{N}$ and any $\pi\in \widehat{\mathbb{G}}.$ To solve the system  given by the infinite matrix equation \eqref{diagon}, we decouple the system by fixing an index $(i,j)$. Then each equation given by \eqref{diagon} consists of functions only in the time-variable $t$ and will be treated independently. Thus, normalizing equation \eqref{diagon} for fixed $(i,j)$ and taking into account that $(k,l)\in \mathcal{PC}$ we have that problem \eqref{diagon} has a nonnegative and nonincreasing solution (see Subsection \ref{oned}), which we denote by $\overline{s_{ij}}_{\pi_i^2}(t)$, and after making it into an infinite matrix with respect to $i,j$, we denote it by $\overline{s}_{\pi}(t)$. Therefore, one has $$\widehat{u}(t,\pi)=\overline{s}_{\pi}(t)\widehat{u_0}(\pi).$$ Now by the application of the inverse Fourier transform and \eqref{Four.inv.for} we arrive at  
\[
u(t,x)=\int_{\widehat{\mathbb{G}}}\textnormal{Tr}[\pi(x)\overline{s}_{\pi}(t)\widehat{u_0}(\pi)]\,{\rm d}\mu(\pi).
\]
On the other hand, we already know that equation \eqref{integral-equivalent} is well-posed in $L^p(\mathbb{G})$ and the solution can be given by  $u(t)=S(t)u_0$ (see the introduction for all details). We note that then by continuity the solution operator $u(t)=S(t)u_0$ to \eqref{integral-equivalent} can be extended continuously to the whole of $L^p(\mathbb{G})$, with values in $L^q(\mathbb{G}).$ Indeed, from \cite[Theorem 5.1]{RR2020}, \cite[Theorem 1.1]{SRR} and Lemma \ref{lemma-decay} it follows that
\begin{align}\label{volterra-sol}
\|u(t)\|_{L^q(\mathbb{G})}&\leqslant C\|u_0\|_{L^p(\mathbb{G})}\|S(t)\|_{L^{r,\infty}(VN_R(\mathbb{G}))} \nonumber\\
&\leqslant  C\|u_0\|_{L^p(\mathbb{G})}\sup_{v>0}\frac{1}{1+v(1\ast l)(t)}\big[\tau(E_{(0,v)}(\mathcal{R}))\big]^{\frac{1}{r}},\quad \frac{1}{r}=\frac{1}{p}
-\frac{1}{q}.
\end{align}
Now, it was shown in \cite[Theorem 8.2]{david} that 
\[
\tau\big(E_{(0,v)}(\mathcal{R})\big)\lesssim v^{Q/\nu},\quad v\to+\infty.
\]
Moreover, in this case the estimate is sharp and we can then claim the sharpness of the time-decay rate. So, by plugging the latter estimate in \eqref{volterra-sol} it yields
\begin{align*}
\|u(t)\|_{L^q(\mathbb{G})}&\leqslant C\|u_0\|_{L^p(\mathbb{G})}\sup_{v>0}\frac{v^{\lambda/r}}{1+v(1\ast l)(t)}
\\&\leqslant C\bigg((1\ast l)(t)\bigg)^{-\lambda/r}\|u_0\|_{L^p(\mathbb{G})},\,\,\lambda=Q/\nu, 
\end{align*}
since the supremum is attained at $v=\frac{\lambda}{(r-\lambda)(1\ast l)(t)}$ whenever $\frac{1}{\lambda}>\frac{1}{p}-\frac{1}{q}.$ Note now that in the case $\frac{1}{\lambda}=\frac{1}{p}-\frac{1}{q}=\frac{1}{r},$ we get straightforward the following decay:
\[
\sup_{v>0}\frac{v^{\lambda/r}}{1+v(1\ast l)(t)}=\sup_{v>0}\frac{v}{1+v(1\ast l)(t)}\leqslant \bigg((1\ast l)(t)\bigg)^{-1}, 
\]
completing the proof.
\end{proof}
    
\begin{remark}
Below we establish a more general result than Theorem \ref{integral-thm} but we need to restrict the source function to be in a specific Sobolev space. We also do not use the Fourier analysis on the group. 
\end{remark}  

First we need to recall the  Sobolev space $W^{1,1}([0,T];L^p(\mathbb{G}))$ $(1<p\leqslant 2)$ that has all functions $g:[0,T]\to L^p(\mathbb{G})$ such that the distributional derivatives satisfy $D^{\beta}g\in L^1([0,T];L^p(\mathbb{G}))$ for orders $|\beta|\leqslant 1.$

As usual in our analysis, we work with the equivalent integral equation of problem \eqref{1}-\eqref{2}. Hence, notice that equation \eqref{1} is equivalent to
\[
(k\ast u^{\prime})(t)+\mathcal{R}u(t)=f(t),\quad t>0,\,x\in\mathbb{G},
\]
hence by using the Sonine condition we have 
\begin{equation}\label{integral-equivalent-gen}
u(t)-u_0+(l\ast\mathcal{R}u)(t)=(l\ast f)(t).    
\end{equation}
By \cite[Theorem 4.2]{Prus} and \cite[Corollary 4.2.9]{FR16}, the above equation admits a resolvent in $L^p(\mathbb{G})$ (exponentially bounded). Thus, equation \eqref{integral-equivalent-gen} is well-posed \cite[Proposition 1.1]{Prus}.

\begin{theorem}\label{integral-thm-1}
Let $(k,l)\in \mathcal{PC}$ be such that $l$ is positive and $l\in BV_{loc}(\mathbb{R}_+).$ Let $\mathbb{G}$ be a graded Lie group of homogeneous dimension $Q$. Let $\mathcal{R}$ be a positive Rockland operator of homogeneous degree $\nu$ on $\mathbb{G}$. If $u_0$ is a continuous function in $L^p(\mathbb{G})$ $(1<p<+\infty)$, $f\in C([0,T];L^p(\mathbb{G}))$ and $l\ast f\in C([0,T];L^p(\mathbb{G}))$ then the integral equation \eqref{integral-equivalent-gen} is well-posed. Also, if $u_0\in L^p(\mathbb{G})$ $(1<p\leqslant 2)$ and $l\ast f\in W^{1,1}([0,T];L^p(\mathbb{G}))$ then the following mild solution of \eqref{integral-equivalent-gen}
\[
u(t)=S(t)u_0+\int_0^t S(t-r)(l\ast f)^{\prime}(s){\rm d}s
\]
is in 
$L^q(\mathbb{G})$ $(2\leqslant q<+\infty).$ Moreover, we get the following time-rate for the latter solution: 
\begin{align*}
\|u(t)\|_{L^q(\mathbb{G})}\lesssim &\bigg((1\ast l)(t)\bigg)^{-Q/\nu\left(\frac{1}{p}-\frac{1}{q}\right)}\|u_0\|_{L^p(\mathbb{G})} \\
&+\int_0^t\bigg((1\ast l)(s)\bigg)^{-Q/\nu\left(\frac{1}{p}-\frac{1}{q}\right)}\|(l\ast f^{\prime})(s)\|_{L^p(\mathbb{G})}{\rm d}s,\quad \frac{\nu}{Q}\geqslant\frac{1}{p}-\frac{1}{q}.   \end{align*}
\end{theorem}
\begin{remark}
We should note that if $(l\ast f)^{\prime}\in L^{\infty}([0,T];L^p(\mathbb{G}))$ then 
\begin{align*}
\|u(t)\|_{L^q(\mathbb{G})}\lesssim &\bigg((1\ast l)(t)\bigg)^{-Q/\nu\left(\frac{1}{p}-\frac{1}{q}\right)}\|u_0\|_{L^p(\mathbb{G})} \\
&+\sup_{s\in [0,T]}\|(l\ast f^{\prime})(s)\|_{L^p(\mathbb{G})}\int_0^t\bigg((1\ast l)(s)\bigg)^{-Q/\nu\left(\frac{1}{p}-\frac{1}{q}\right)}{\rm d}s,\quad \frac{\nu}{Q}\geqslant\frac{1}{p}-\frac{1}{q}.   \end{align*}
\end{remark}
\begin{proof}[Proof of Theorem \ref{integral-thm-1}]
By \cite[Proposition 1.2, item (ii)]{Prus} and condition 

$l\ast f\in W^{1,1}([0,T];L^p(\mathbb{G}))$ we obtain that the solution of \eqref{integral-equivalent-gen} is given as
\[
u(t)=S(t)u_0+\int_0^t S(t-r)(l\ast f)^{\prime}(s){\rm d}s.
\]
From \cite[Theorem 5.1]{RR2020}, \cite[Theorem 1.1]{SRR} and Lemma \ref{lemma-decay} it follows that
\begin{align}\label{volterra-sol}
&\|u(t)\|_{L^q(\mathbb{G})}
 \nonumber\\ &\leqslant C\left(\|u_0\|_{L^p(\mathbb{G})}\|S(t)\|_{L^{r,\infty}(VN_R(\mathbb{G}))}+\int_0^t \|S(t-r)\|_{L^{r,\infty}(VN_R(\mathbb{G}))}\|(l\ast f)^{\prime}(s)\|_{L^p(\mathbb{G})}{\rm d}s \right)\nonumber\\
&\leqslant  C\bigg(\|u_0\|_{L^p(\mathbb{G})}\sup_{v>0}\frac{\big[\tau(E_{(0,v)}(\mathcal{R}))\big]^{\frac{1}{r}}}{1+v(1\ast l)(t)}+\int_0^t \|(l\ast f)^{\prime}(s)\|_{L^p(\mathbb{G})}\sup_{v>0}\frac{\big[\tau(E_{(0,v)}(\mathcal{R}))\big]^{\frac{1}{r}}}{1+v(1\ast l)(s)}{\rm d}s\bigg),\quad\nonumber
\end{align}
where $\frac{1}{r}=\frac{1}{p}
-\frac{1}{q}.$ Now we just follow the same lines from equation \eqref{volterra-sol} of the proof on Theorem \ref{integral-thm}. The proof is complete.
\end{proof}

\subsection{Estimates in a homogeneous Sobolev space}

Notice that in the last two theorems we were not able to gain any decay in the $L^2-$norm. This mainly happens since for the trivial representation on $\widehat{\mathbb{G}}$ the spectrum comes to $0$. Now we recall a homogeneous Sololev space that involves the Rockland operator. In this space we can avoid somehow the contribution of the trivial representation and gain some time-decay in problem \eqref{1}-\eqref{2}. For a detailed discussion on the Sobolev spaces in the setting of a graded Lie group, see \cite[Subsection 4.4.1]{FR16} and \cite{FR17}. 

Let $\mathcal{R}$ be a positive Rockland operator of homogeneous degree $\nu$, and let $s\in\mathbb{R}$. The \textit{homogeneous Sobolev space $\dot{L}^{2}_{s}(\mathbb{G})$} is the subspace of tempered distributions $\mathcal{S}'(\mathbb{G})$ obtained by the completion of $ \mathcal{S}(\mathbb{G})$ with respect to the (homogeneous) \textit{Sobolev norm}
\[
\|u\|_{\dot{L}^{2}_{s}(\mathbb{G})}:=\|\mathcal{R}_{2}^{\frac{s}{\nu}}u\|_{L^2(\mathbb{G})}\,,\quad u \in \mathcal{S}(\mathbb{G}),
\]
where the operator $\mathcal{R}_2$ stands for the self-adjoint extension of $\mathcal{R}$ on $L^2(\mathbb{G})$.

\begin{theorem}\label{integral-thm-sobolev}
Let $(k,l)\in \mathcal{PC}$ be such that $l$ is positive and $l\in BV_{loc}(\mathbb{R}_+).$ Let $\mathbb{G}$ be a graded Lie group of homogeneous dimension $Q$. Let $\mathcal{R}$ be a positive Rockland operator of homogeneous degree $\nu$ on $\mathbb{G}$. For any $s\geqslant v$, if $u_0\in \dot{L}^2_{s-\nu}(\mathbb{G})$ then the solution of integro-differential difussion equation \eqref{1}-\eqref{2} $(f\equiv0)$ belongs to $\dot{L}^2_{s}(\mathbb{G})$ and satisfies
\[
\|u(t)\|_{\dot{L}^2_{s}(\mathbb{G})}\lesssim \bigg((1\ast l)(t)\bigg)^{-1}\|u_0\|_{\dot{L}^2_{s-\nu}(\mathbb{G})}.    
\]
\end{theorem}
\begin{proof}
From the first part of the proof of Theorem \ref{integral-thm} we have that $$\widehat{u}(t,\pi)_{ij}=\overline{s_{ij}}_{\pi_i^2}(t)\widehat{u_0}(\pi)_{ij}$$ for any $i,j.$ By estimate \eqref{1.3} we obtain
\begin{align*}
\pi_{i}^{2s/\nu}|\widehat{u}(t,\pi)_{i,j}|&\lesssim \frac{\pi_{i}^{2s/\nu}}{1+\pi_{i}^2 (1\ast l)(t)}|\widehat{u}_{0}(\pi)_{i,j}| \\
&\lesssim \big((1\ast l)(t)\big)^{-1}\pi_{i}^{\big(2s/\nu\big)-2}|\widehat{u}_{0}(\pi)_{i,j}|,\quad \text{for all} \quad t>0.
\end{align*}
Thus, after summation over $i,j$, integration over $\widehat{\mathbb{G}}$ with respect to the Plancherel measure and using the Plancherel formula \eqref{plan.gr} we finally get 
\[
\|u(t)\|_{\dot{L}^2_s(\mathbb{G})}\lesssim \big((1\ast l)(t)\big)^{-1}\|u_0\|_{\dot{L}^2_{s-\nu}(\mathbb{G})}, \quad \text{for all} \quad t>0,
\]
completing the proof.
\end{proof}
We finish this section with a regularity result for equation \eqref{1} in the Sobolev space.

\begin{theorem}\label{integral-thm-1-sobolev}
Let $(k,l)\in \mathcal{PC}$ be such that $l$ is positive and $l\in BV_{loc}(\mathbb{R}_+).$ Let $\mathbb{G}$ be a graded Lie group of homogeneous dimension $Q$. Let $\mathcal{R}$ be a positive Rockland operator of homogeneous degree $\nu$ on $\mathbb{G}$. For any $s\geqslant v$, if $u_0\in \dot{L}^2_{s-\nu}(\mathbb{G})$ and $\bigg(\frac{l}{(1\ast l)(\cdot)}\ast f\bigg)(t)\in \dot{L}^2_{s-\nu}(\mathbb{G})$ then the solution of the integro-differential difussion equation \eqref{1}-\eqref{2} belongs to $\dot{L}^2_s(\mathbb{G})$ and satisfies
\[
\|u(t)\|_{\dot{L}^2_s(\mathbb{G})}\lesssim \bigg((1\ast l)(t)\bigg)^{-1}\|u_0\|_{\dot{L}^2_{s-\nu}(\mathbb{G})}+\int_0^t\left\|\bigg(\frac{l}{(1\ast l)(\cdot)}\ast f\bigg)(s)\right\|_{\dot{L}^2_{s-\nu}(\mathbb{G})}{\rm d}s.
\]
\end{theorem}
\begin{proof}
Applying the Fourier transform to equation \eqref{1} and taking the components of the infinite system we have 
\begin{align*}
\begin{split}
\partial_t\left(k\ast\left[\widehat{u}(t,\pi)_{ij}-\widehat{u_0}(\pi)_{ij}\right]\right)+\pi_{i}^{2}\widehat{u}(t,\pi)_{ij}&=\widehat{f}(t,\pi)_{ij},\quad \widehat{u}(t,\pi)_{ij}|_{_{_{t=0}}}=\widehat{u_0}(\pi)_{ij};\quad t>0,
\end{split}
\end{align*}
where we are considering any $i,j\in\mathbb{N}$ and any $\pi\in \widehat{\mathbb{G}}.$ By \cite[Lemma
5.1 and Lemma 5.3]{[15]}, the solution is given by 
\[
\widehat{u}(t,\pi)_{ij}=\overline{s_{ij}}_{\pi_{i}^2}(t)\widehat{u_0}(\pi)_{ij}+\int_0^t \widehat{f}(t-h,\pi)_{ij}\overline{r_{ij}}_{\pi_{i}^2}(h){\rm d}h,
\]
where $\overline{r_{ij}}_{\pi_{i}^2}$ is a nonnegative and nonincreasing function. Moreover, by estimate \eqref{1.3} and \cite[Formula (5.9)]{[15]} we arrive at
\begin{align*}
\pi_{i}^{2s/\nu}|\widehat{u}(t,\pi)_{i,j}|&\lesssim \frac{\pi_{i}^{2s/\nu}}{1+\pi_{i}^2 (1\ast l)(t)}|\widehat{u}_{0}(\pi)_{i,j}|+\pi_{ij}^{2s/\nu}\int_0^t \widehat{f}(t-h,\pi)_{ij}\frac{l(h)}{1+\pi_{i}^2 (1\ast l)(h)}{\rm d}h \\
&\lesssim \frac{\pi_{ij}^{\big(2s/\nu\big)-2}}{(1\ast l)(t)}|\widehat{u}_{0}(\pi)_{i,j}|+\pi_{ij}^{\big(2s/\nu\big)-2}\int_0^t \widehat{f}(t-h,\pi)_{ij}\frac{l(h)}{(1\ast l)(h)}{\rm d}h.
\end{align*}
From the above inequality we can then get the desired result.
\end{proof}

\subsection{Remarks about some other similar type of equations}

\begin{remark}
Notice that same techniques and tools as those used in this section can be applied to get similar results for the following generalized subdiffusion equation:
\begin{equation}\label{sub-di}
\partial_t u-\partial_t(l\ast \mathcal{R}u)(t)=f,\quad t>0,    
\end{equation}
with initial data $u_0.$

These type of equations were studied in \cite{otro}. We first highlight that the homogeneous equation \eqref{sub-di} is equivalent to \eqref{integral-equivalent}. Therefore, we do not consider now this case. Let us then focus in the non homogeneous case, i.e. equation \eqref{sub-di}. Note now that equation \eqref{sub-di} can be rewritten as 
\begin{equation}\label{integral-equivalent-gen-2}
u(t)-u_0+(l\ast\mathcal{R}u)(t)=\int_0^t f(s){\rm d}s.    
\end{equation}
So, equation \eqref{integral-equivalent-gen-2} is similar to \eqref{integral-equivalent}. The only difference is the source function involved. Thus, we give the statement without proof since it is the same as given in Theorem \ref{integral-thm-1}.

\begin{theorem}\label{integral-thm-2}
Let $(k, l)\in \mathcal{PC}$ and $l\in BV_{loc}(\mathbb{R}_+).$ Let $\mathbb{G}$ be a graded Lie group of homogeneous dimension $Q$. Let $\mathcal{R}$ be a positive Rockland operator of homogeneous degree $\nu$ on $\mathbb{G}$. If $u_0$ and $f$ are continuous functions in $L^p(\mathbb{G})$ $(1<p<+\infty)$ then the integral equation \eqref{integral-equivalent-gen-2} is well-posed. Also, if $u_0\in L^p(\mathbb{G})$ $(1<p\leqslant 2)$ and $f\in L^{\infty}([0,T],L^p(\mathbb{G}))$ then the following mild solution of \eqref{integral-equivalent-gen-2}
\[
u(t)=S(t)u_0+\int_0^t S(t-r)f(s){\rm d}s
\]
is in 
$L^q(\mathbb{G})$ $(2\leqslant q<+\infty).$ Moreover, we get the following time-rate for the latter solution: 
\begin{align*}
\|u(t)\|_{L^q(\mathbb{G})}\lesssim &\bigg((1\ast l)(t)\bigg)^{-Q/\nu\left(\frac{1}{p}-\frac{1}{q}\right)}\|u_0\|_{L^p(\mathbb{G})} \\
&+\| f\|_{L^{\infty}([0,T],L^p(\mathbb{G}))}\int_0^t\bigg((1\ast l)(s)\bigg)^{-Q/\nu\left(\frac{1}{p}-\frac{1}{q}\right)}{\rm d}s,\quad \frac{\nu}{Q}\geqslant\frac{1}{p}-\frac{1}{q}.   
\end{align*}
\end{theorem}

\end{remark}

\section{Examples of Rockland operators}\label{roc}

\subsection{Euclidean space}
Let $\mathbb{G} = \mathbb{R}^n.$ Then the Rockland operator $\mathcal{R}$ in $\mathbb{R}^n$ may be any positive homogeneous elliptic operators with constant coefficients. For example:
\begin{itemize}
    \item Laplace operator $$-\Delta=-\sum\limits_{j=1}^n\frac{\partial^2}{\partial x^2_j};$$
    \item Poly-harmonic operator $$(-\Delta)^m=\underbrace{(-\Delta)\cdot\cdots\cdot(-\Delta)}_{m},\quad,m\in\mathbb{N};$$
    \item Higher order elliptic operator
$$(-1)^m\sum\limits_{j=1}^na_j\left(\frac{\partial}{\partial x_j}\right)^{2m},$$ where $a_j>0$ and $m\in\mathbb{N}.$
\end{itemize}
\subsection{Heisenberg group $\mathbb{H}^n$.} The Heisenberg group $\mathbb{H}^n$ is the manifold $\mathbb{R}^{2n+1}$ endowed with the composition law
\[(x,y,t)(x',y',t'):=(x+x',y+y',t+t'+\frac{1}{2}(xy'-x'y))\,,\]
where $(x,y,t),(x',y',t') \in \mathbb{H}^n$. The canonical basis of its Lie algebra is given by the (left-invariant) vector fields
        \[
        X_j=\partial_{x_j}-\frac{y_j}{2}\partial_t\,,\quad Y_j=\partial_{y_j}+\frac{x_j}{2}\partial_t\,,\quad \text{and}\quad T=\partial_t\,.
        \]
       \begin{itemize}
       \item The sub-Laplacian in this setting is the simplest example of a Rockland operator and is given by
        \[
        -\Delta_{\mathbb{H}}=-\sum_{j=1}^{n}(X_{j}^{2}+Y_{j}^{2})\,.
        \]
       \item  More generally a positive Rockland operator in this setting is given by 
        \[
        \mathcal{R}=\sum_{i=1}^{2n+1}(-1)^{\frac{\nu_0}{\nu_i}}c_i A_{i}^{2\frac{\nu_0}{\nu_i}}\,,\quad c_i>0\,,
        \]
        where $\nu_i \in \{1,2\}$, $\nu_0$ is an even number, and $A_i \in \{X_j,Y_j,T \} $, for all $j=1,\cdots, n$ and $i=1,\cdots, 2n+1$. For the homogeneous dimension $Q$ we have $Q=2n+2$.
\end{itemize}

\subsection{Engel group $\mathcal{B}_4$.} The Engel group $\mathcal{B}^4$ is the manifold $\mathbb{R}^{4}$ endowed with the composition law
       \begin{eqnarray*}
\lefteqn{(x_1,x_2,x_3,x_4) \times (y_1,y_2,y_3,y_4)}\\
&:=& (x_1+y_1,x_2+y_2,x_3+y_3-x_1y_2,x_4+y_4+\frac{1}{2}x_1^2y_2-x_1y_3)\,.
\end{eqnarray*}
         The canonical basis of its Lie algebra is given by the (left-invariant) vector fields
      \begin{equation*}\label{left.inv.eng}
\begin{split}
X_1(x)&= \frac{\partial}{\partial x_1}\,, \quad X_2(x)=\frac{\partial}{\partial x_2}-x_1 \frac{\partial}{\partial x_3}+ \frac{x^{2}_{1}}{2}\frac{\partial}{\partial x_4}\,, \\
            X_3(x)&= \frac{\partial}{\partial x_3}-x_1 \frac{\partial}{\partial x_4}\,,\quad X_4(x)= \frac{\partial}{\partial x_4}\,,
\end{split}
\end{equation*}
for $x=(x_1,x_2,x_3,x_4) \in \mathcal{B}^4=\mathbb{R}^4$. 
\begin{itemize}
    \item The sub-Laplacian in this setting is the simplest example of a Rockland operator and is given by
        \[
        -\Delta_{\mathcal{B}}=-\sum_{j=1}^{n}(X_{1}^{2}+X_{2}^{2})\,.
        \]
    \item More generally a positive Rockland operator in this setting is given by 
        \[
        \mathcal{R}=\sum_{i=1}^{4}(-1)^{\frac{\nu_0}{\nu_i}}c_i X_{i}^{2\frac{\nu_0}{\nu_i}}\,,\quad c_i>0,
        \]
        where $\nu_i \in \{1,2,3\}$ and $\nu_0$ is an even number that is a multiple of $3$.  For the homogeneous dimension $Q$ we have $Q=7$.
\end{itemize}

\section{Particular cases and comparisons with known results}\label{comparison}

Below we denote by $g_{\mu}(t)=\frac{t^{\mu-1}}{\Gamma(\mu)},$ for $t,\mu>0$ where $\Gamma(\cdot)$ is the Euler gamma function.

\subsection{Time-fractional diffusion equations} The most classical example of a Sonine pair $(k,l)$ is $k(t)=g_{1-\alpha}(t)$ and $l(t)=g_{\alpha}(t)$ with $0<\alpha<1$. Notice now that $\partial_t (k\ast u)(t)$ coincides with the Riemann-Liouville time-fractional derivative $$D^\alpha_tu(t):=\partial_t (g_{1-\alpha}\ast u)(t)=\frac{1}{\Gamma(1-\alpha)}\frac{\rm d}{\rm{d}t}\int\limits_0^t(t-s)^{-\alpha}u(s){\rm d}s$$ of order $\alpha\in(0,1)$. Hence, equation \eqref{1} $(f\equiv0)$ goes over to the following time-fractional diffusion equation
\begin{equation}\label{1-1}
D^\alpha_t\left(u(t)-u_0\right)+\mathcal{R}u(t)=0,\quad t>0,\quad u_0\in L^p(\mathbb{G}).
\end{equation}
Then, according to Theorem \ref{integral-thm}, the solution of problem \eqref{1-1} with condition \eqref{2} satisfies
\begin{align*}
\|u(t)\|_{L^q(\mathbb{G})}\lesssim &t^{-\frac{\alpha Q}{\nu}\left(\frac{1}{p}-\frac{1}{q}\right)}\|u_0\|_{L^p(\mathbb{G})},\quad \frac{\nu}{Q}\geqslant\frac{1}{p}-\frac{1}{q}.  
\end{align*}
Then the above $L^p(\mathbb{G})\,-\,L^q(\mathbb{G})$ estimate coincides with the estimate obtained in \cite[Theorem 3.3]{SRR} for the problem \eqref{1-1}. 

Also, when $\mathbb{G}=\mathbb{R}^n$ and $\mathcal{R}=-\Delta$, then the estimate above coincides with the estimate
\begin{align*}
\|u(t)\|_{L^q(\mathbb{R}^n)}\lesssim &t^{- \frac{\alpha n}{2}\left(\frac{1}{p}-\frac{1}{q}\right)}\|u_0\|_{L^p(\mathbb{R}^n)},\quad \frac{2}{n}\geqslant\frac{1}{p}-\frac{1}{q},  
\end{align*}
obtained in \cite[Theorem 3.3]{Zacher1}.

In the case of the Heisenberg group, that is, if $\mathbb{G}=\mathbb{H}^n$ and $\mathcal{R}=-\Delta_{\mathbb{H}^n}$, the estimate \eqref{lp} has the form
\begin{align*}
\|u(t)\|_{L^q(\mathbb{H}^n)}\lesssim &t^{-\frac{\alpha Q}{2}\left(\frac{1}{p}-\frac{1}{q}\right)}\|u_0\|_{L^p(\mathbb{H}^n)},\quad \frac{2}{Q}\geqslant\frac{1}{p}-\frac{1}{q},  \end{align*}
where $-\Delta_{\mathbb{H}^n}$ is the sub-Laplacian on $\mathbb{H}^n$ and $Q=2n+2$ is the homogeneous dimension of the Heisenberg group $\mathbb{H}^n$.

\subsection{Two term time-fractional diffusion equations} Let
$$k(t)=g_{1-\beta+\alpha}(t)+g_{1-\beta}(t),\quad l(t)=t^{\beta-1}E_{\alpha,\beta}(-t^\alpha),\quad 0<\alpha<\beta<1,$$ where $E_{\alpha,\beta}(y)=\sum\limits_{k=0}^{+\infty}\frac{(-y)^{k}}{\Gamma(\alpha k+\beta)}$ is the two parametric Mittag-Leffler function. Since 
\[
(1\ast l)(t)=\int\limits_0^ts^{\beta-1}E_{\alpha,\beta}(-s^\alpha){\rm d}s=t^{\beta}E_{\alpha,\beta+1}(-t^\alpha),
\]
Theorem \ref{integral-thm} implies that the solution $u(t)$ of the following equation 
\[
D^{\beta}_t\left(u(t)-u_0\right)+D^{\beta-\alpha}_t\left(u(t)-u_0\right)+\mathcal{R}u(t)=0,\quad t>0,\quad u_0\in L^p(\mathbb{G}),
\]
satisfies that 
\[
\|u(t)\|_{L^q(\mathbb{G})}\lesssim \,\big(t^{\beta}E_{\alpha,\beta+1}(-t^\alpha)\big)^{-\frac{Q}{\nu}\left(\frac{1}{p}-\frac{1}{q}\right)}\|u_0\|_{L^p(\mathbb{G})},\quad t>0.
\]

\subsection{Distributed order diffusion equations}
Let $(k,l)$ be defined as follows $$k(t)=\int\limits_0^1g_\alpha(t){\rm d}\alpha,\quad  l(t)=\int\limits_0^{+\infty} \frac{e^{-st}}{1+s}{\rm d}s,\quad t>0.$$
In this case $\partial_t (k\ast u)(t)$ coincides with the distributed order derivative $$\mathcal{D}^{(\mu)}_tu(t)=\int\limits_0^1(D^\alpha u(t))\mu(\alpha){\rm d}\alpha,$$ where the non-negative weight function $\mu \in C[0, 1]$ does not vanish on the interval $[0, 1].$ Let us now consider the following ultraslow (distributed order) diffusion equation:
\begin{equation*}
\mathcal{D}^{(\mu)}_t\left(u(t)-u_0\right)+\mathcal{R}u(t)=0,\quad t>0,\quad u_0\in L^p(\mathbb{G}).
\end{equation*}
Therefore, for the solution of equation above, the estimate \eqref{lp} takes the form 
\begin{align}\label{lp_d-1}
\|u(t)\|_{L^q(\mathbb{G})}\lesssim &\bigg(\int_0^t \left[\int\limits_0^{+\infty} \frac{e^{-sh}}{1+s}{\rm d}s\right]{\rm d}h\bigg)^{-\frac{Q}{\nu}\left(\frac{1}{p}-\frac{1}{q}\right)}\|u_0\|_{L^p(\mathbb{G})}, \nonumber\\
=&\bigg(\int_0^t e^{h}Ei(-h){\rm d}h\bigg)^{-\frac{Q}{\nu}\left(\frac{1}{p}-\frac{1}{q}\right)}\|u_0\|_{L^p(\mathbb{G})},\quad t>0,
\end{align}
where $Ei(h)$ is the exponential integral function. As far as we know, estimate \eqref{lp_d-1} gives a new result even in case $\mathbb{G}=\mathbb{R}^n$.

\subsection{Multi-term time-fractional diffusion equation}
Let $$k(t)=\sum\limits_{j=1}^mg_{1-\alpha_j}(t),\quad 0<\alpha_m<\cdots<\alpha_2<\alpha_1<1.$$
Then equation \eqref{1} is reduced to the multi-term time-fractional diffusion equation
\begin{equation}\label{1mt}
\sum\limits_{j=1}^m D^{\alpha_j}_t\left(u(t)-u_0\right)+\mathcal{R}u(t)=0,\quad t>0,\quad u_0\in L^p(\mathbb{G}).
\end{equation}
Note that the equation \eqref{1mt} with non-local time conditions was studied in \cite{ruzh}.

The solution of the multi-term time-fractional diffusion equation \eqref{1mt} is represented using the multivariate Mittag-Leffler function 
$$E_{(\alpha_1,...,\alpha_{m}),\beta}(z_1,\ldots,z_{m})
=\sum\limits_{k=0}^{+\infty}\sum\limits_{\begin{array}{l}l_1+\ldots+l_{m}=k,\\
l_1\geqslant 0,\ldots,l_{m}\geqslant 0\end{array}}\frac{k!}{l_1!\ldots l_{n}!}\frac{\prod\limits_{j=1}^{m} z_j^{l_j}}{\Gamma\left(\beta+\sum\limits_{j=1}^{m}\alpha_jl_j\right)}.$$
In \cite[Lemma 3.2]{yamam} it is proven that if $\mu\leqslant |\arg(z_1)|\leqslant \pi,\,\alpha_1\pi/2<\mu<\alpha_1\pi$ and there exists $K > 0$ such that $-K\leqslant z_j<0\,\, (j=2,...,m)$, then there exists a constant $C > 0$ depending only on $\mu,$ $K$, $\alpha_j\,(j=1,...,m)$
and $\beta$ such that
$$|E_{(\alpha_1,\alpha_1-\alpha_2,\dots,\alpha_1-\alpha_{m}),\beta}(z_1,\ldots,z_{m})|\leqslant \frac{C}{1+|z_1|},$$ where $0<\alpha_m<\cdots<\alpha_2<\alpha_1<1$ and $0<\beta<2.$

Unfortunately, this estimate is not suitable for asympototic estimates at large $t$.

Nevertheless, by using the estimate \eqref{lp} we arrive at 
\begin{align*}
\|u(t)\|_{L^q(\mathbb{G})}\lesssim\big(t^{\alpha_1-1}E_{(\alpha_1-\alpha_2,\ldots,\alpha_1-\alpha_m),\alpha_1}(-t^{\alpha_1-\alpha_2},\ldots,-t^{\alpha_1-\alpha_m})\big)^{-\frac{Q}{\nu}\left(\frac{1}{p}-\frac{1}{q}\right)}\|u_0\|_{L^p(\mathbb{G})},
\end{align*}
which gives an estimate of the solution of equation \eqref{1mt}.

This result improves decay estimates from \cite{ruzh}, where decay estimates of local in time solutions were obtained with the time-dependent constants.

\subsection{Tempered fractional diffusion equations}
Let 
$$k(t)=g_{1-\alpha}(t)e^{-\gamma t},\quad l(t)=g_\alpha(t)e^{-\gamma t}+\gamma\int\limits_0^tg_\alpha(s)e^{-\gamma s}{\rm d}s,\quad t>0,$$ where $\alpha\in (0,1),$ $\gamma>0,$ then $$\partial_t (k\ast u)(t)=\frac{1}{\Gamma(1-\alpha)}\partial_t\int\limits_0^t(t-s)^{-\alpha}e^{-\gamma(t-s)}u(s){\rm d}s$$ coincides with the tempered fractional derivative, see e.g. \cite{Vergara1}. Then, the equation \eqref{1} goes over to the tempered fractional diffusion equation
\begin{equation}\label{1-tem}
\partial_t \left[g_{1-\alpha}(t)e^{-\gamma t}\ast \left(u(t)-u_0\right)\right]+\mathcal{R}u(t)=0,\quad t>0,\quad u_0\in L^p(\mathbb{G}).
\end{equation}
Hence, the estimate \eqref{lp} takes the form 
\begin{align*}
\|u(t)\|_{L^q(\mathbb{G})}\lesssim\, &\bigg(\int_0^t g_\alpha(h)e^{-\gamma h}{\rm d}h+\int_0^t\bigg[\gamma\int\limits_0^h g_\alpha(s)e^{-\gamma s}{\rm d}s\bigg]{\rm d}h\bigg)^{-\frac{Q}{\nu}\left(\frac{1}{p}-\frac{1}{q}\right)}\|u_0\|_{L^p(\mathbb{G})},
\end{align*}
which gives the $L^p(\mathbb{G})-L^q(\mathbb{G})$-estimate of solution to the tempered diffusion equation \eqref{1-tem}.

\section{Equations on compact Lie groups}

Note that the ideas and methods in this paper can be applied to study some similar equations on $L^p(\mathrm{G})$ when $\mathrm{G}$ is a compact Lie group. Basic steps, tools and techniques of more simple equations (i.e. $k(t)=t^{-\alpha}$, $0<\alpha<1$, $t\geqslant0$) in a compact Lie group can be found in \cite{wagner}. The type of operators that can be in principle considered are any positive linear left invariant operator acting on $\mathrm{G},$ such that it can be well-defined in $L^p(\mathrm{G})$ having a dense domain there. So, we can study the following equation on $L^p(\mathrm{G})$: 
\begin{equation*}
\left\{\begin{aligned}
		\partial_t\left(k\ast\left[u-u_0\right]\right)&+\mathcal{L}u=f,\quad 0<t\leqslant T<+\infty, \\
	u_0&\in C(\mathrm{G})\cap L^p(\mathrm{G}),
	\end{aligned}
	\right.
\end{equation*}
where $k$ is a scalar kernel $\neq0$ in $L^1_{loc}(\mathbb{R}_+)$ that belongs to the class $\mathcal{PC}$, $\mathcal{L}$ is a positive linear left invariant operator on $\mathrm{G}$ (maybe unbounded) and $f\in C([0,T],L^p(\mathrm{G}))$. Equation above is equivalent to 
\begin{equation}\label{Heat-intro}
\left\{\begin{aligned}
		u-u_0+l\ast\mathcal{L}u&=l\ast f,\quad 0<t\leqslant T<+\infty, \\
	u_0&\in C(\mathrm{G})\cap L^p(\mathrm{G}).
	\end{aligned}
	\right.
\end{equation}

Let us now provide the statements of the analogous results for equation \eqref{Heat-intro}.

\begin{theorem}
Let $(k,l)\in \mathcal{PC}$ be such that $l$ is positive and $l\in BV_{loc}(\mathbb{R}_+).$ Let $\mathrm{G}$ be a compact Lie group. Let $\mathcal{L}$ be a positive linear left invariant operator on $\mathrm{G}$ (maybe unbounded) such that 
	\begin{equation}\label{need}
		\sup_{t>0}\sup_{s>0}\frac{[\tau\big(E_{(0,s)}(\mathcal{L})\big)]^{\frac{1}{p}-\frac{1}{q}}}{1+s(1\ast l)(t)}<+\infty,\quad 1\leqslant p\leqslant 2\leqslant q<+\infty.   
	\end{equation}
If $u_0$ is a continuous function in $L^p(\mathrm{G})$ $(1<p<+\infty)$ then the homogeneous integral equation \eqref{Heat-intro} $(f\equiv0)$ is well-posed. In particular, if for some $\lambda>0$ we have 
\begin{equation}\label{asymtotic-trace}
		\tau\big(E_{(0,s)}(\mathcal{L})\big)\lesssim s^{\lambda},\quad s\to+\infty,
	\end{equation}
	then \eqref{need} is satisfied for any $1<p\leqslant 2\leqslant q<+\infty$ such that $\frac{1}{\lambda}\geqslant\frac{1}{p}-\frac{1}{q}$, and one has the following time decay rate for the solution:  
\begin{equation*}
\|u(t)\|_{L^q(\mathrm{G})}\lesssim \bigg((1\ast l)(t)\bigg)^{-\lambda\left(\frac{1}{p}-\frac{1}{q}\right)}\|u_0\|_{L^p(\mathrm{G})},\quad \frac{1}{\lambda}\geqslant\frac{1}{p}-\frac{1}{q}.    
\end{equation*}
\end{theorem}

\begin{theorem}
Let $(k,l)\in \mathcal{PC}$ be such that $l$ is positive and $l\in BV_{loc}(\mathbb{R}_+).$ Let $\mathrm{G}$ be a compact Lie group. Let $\mathcal{L}$ be a positive linear left invariant operator on $\mathrm{G}$ (maybe unbounded) such that condition \eqref{need} is true. If $u_0$ is a continuous function in $L^p(\mathrm{G})$ $(1<p<+\infty)$, $f\in C([0,T];L^p(\mathrm{G}))$ and $l\ast f\in C([0,T];L^p(\mathrm{G}))$ then the integral equation \eqref{Heat-intro} is well-posed. Also, if the condition \eqref{asymtotic-trace} holds, $u_0\in L^p(\mathrm{G})$ $(1<p\leqslant 2)$ and $l\ast f\in W^{1,1}([0,T];L^p(\mathrm{G}))$ then the following mild solution of \eqref{Heat-intro}
\[
u(t)=S(t)u_0+\int_0^t S(t-r)(l\ast f)^{\prime}(s){\rm d}s
\]
is in 
$L^q(\mathrm{G})$ $(2\leqslant q<+\infty).$ Moreover, we get the following time rate for the latter solution: 
\begin{align*}
\|u(t)\|_{L^q(\mathrm{G})}\lesssim &\bigg((1\ast l)(t)\bigg)^{-\lambda\left(\frac{1}{p}-\frac{1}{q}\right)}\|u_0\|_{L^p(\mathrm{G})} \\
&+\int_0^t\bigg((1\ast l)(s)\bigg)^{-\lambda\left(\frac{1}{p}-\frac{1}{q}\right)}\|(l\ast f^{\prime})(s)\|_{L^p(\mathrm{G})}{\rm d}s,\quad \frac{1}{\lambda}\geqslant\frac{1}{p}-\frac{1}{q}.   \end{align*}
\end{theorem}

\begin{remark}
Note that if $\mathcal{L}$ is the Laplacian on a compact Lie group of topological dimension $n$, then $\lambda=n/2$ in \eqref{asymtotic-trace}. While, if $\mathcal{L}$ is the sub-Laplacian on a compact Lie group $\mathrm{G}$, then $\lambda=Q/2$ where $Q$ is the Hausdorff dimension of $\mathrm{G}$ with respect to H\"ormander vectors fields in $\mathcal{L}$, see \cite{RR2020}.     
\end{remark}

\end{document}